\documentclass[12pt]{article}
\usepackage[english]{babel}
\usepackage{amsmath,amsthm,amsfonts,amssymb,epsfig}
\usepackage[left=1in,top=1in,right=1in]{geometry}

\newtheorem{thm}{Theorem}[section]
\newtheorem{lem}[thm]{Lemma}
\newtheorem{prop}[thm]{Proposition}
\newtheorem{cor}[thm]{Corollary}

\newcommand{\e}{\varepsilon}

\begin{document}

\title{Absence of site percolation at criticality in $\mathbb{Z}^2 \times \{0,1\}$}
\author{Michael Damron \thanks{The research of M. D. is supported by an NSF postdoctoral fellowship and NSF grants DMS-0901534 and DMS-1007626.} \\ \small{Princeton University} \and Charles M. Newman \thanks{The research of C. M. N. is supported by NSF grants OISE-0730136, DMS-1007524 and DMS-1007626.} \\ \small{Courant Institute, NYU}\and Vladas Sidoravicius \thanks{The research of V. S. is supported by CNPq grants PQ 308787/2011-0 and 484801/2011-2.}\\ \small{IMPA}}

\date{\today}

\maketitle

\abstract{In this note we consider site percolation on a two dimensional sandwich of thickness two, the graph $\mathbb{Z}^2 \times \{0,1\}$. We prove that there is no percolation at the critical point. The same arguments are valid for a sandwich of thickness three with periodic boundary conditions. It remains an open problem to extend this result to other sandwiches.}

\section{Introduction}

Motivated by the well-known open problem to prove that in Bernoulli percolation -- bond or site -- in $\mathbb{Z}^d$, there is no percolation at the critical point for all $d$ (proven so far for $d\leq 2$ and $d \geq 19$), we consider the case of percolation on 2-dimensional slabs, ``sandwiches,'' of the form $\mathbb{Z}^2 \times \{0,\ldots, k-1\}$. We point out in this note that in at least one case, site percolation with $k=2$, this issue can be resolved.

The set of vertices is $S=\mathbb{Z}^2 \times \{0,1\}$ and two vertices are neighbors if they are (three-dimensional) Euclidean distance one from each other. To show that there is no percolation at the critical point, the main tool is the following observation, which we call the \emph{thin sandwich property}. If two paths in $S$ have the property that their projections onto $\mathbb{Z}^2$ intersect, then they must connect in $S$. The thin sandwich property is no longer valid for sandwiches of width bigger than 2 nor for bond percolation on the sandwich of width 2. However, there are other graphs for which it holds. For instance, take a sandwich of thickness 3 with periodic boundary conditions; that is, the set $\mathbb{Z}^2 \times \{0, 1, 2\}$ and declare that any two sites of the form $(x,y,0)$ and $(x,y,2)$ are also neighbors. Then the thin sandwich property holds and the conclusions of the present paper are valid.

\subsection{Main result and notation}

A configuration $\omega$ is in the set $\{0,1\}^S$ and for each $v \in S$, $\omega(v)$ represents the status of the vertex $v$: open if $\omega(v)=1$ and closed otherwise. An open path in $S$ from $v$ to $w$ is a sequence of vertices $(v_0, \ldots, v_n)$ such that $v_0=v$, $v_n=w$, $\|v_i-v_{i+1}\|_1=1$ for all $i = 0, \ldots, n-1$, and all $v_i$'s are open. We say $v$ is connected to infinity $(v \to \infty)$ if $v$ is part of an infinite open path. If this occurs for some $v$, we say percolation occurs. Let $\mathbb{P}_p$ be the product measure on $\{0,1\}^S$ with parameter $p$:
\[
\mathbb{P}_p(\omega(v)=1) = p = 1-\mathbb{P}_p(\omega(v)=0)\ .
\]
There is a critical value $p_c = \sup\{p:\mathbb{P}_p(\vec 0 \to \infty) = 0\}$, where $\vec 0 = (0,0,0)$. 

\begin{thm}\label{thm: mainthm1}
There is no percolation at the critical point:
\[
\theta(p_c) = 0\ ,
\]
where $\theta(p) = \mathbb{P}_p(\vec 0 \to \infty)$.
\end{thm}

We will see that this is a consequence of the following ``finite-volume criterion.'' For $n,m \geq 1$ let $E(n,m)$ be the left-right crossing event that there exists a vertex in $\{-n\} \times [-m,m] \times \{0,1\}$ connected to a vertex in $\{n\} \times [-m,m] \times \{0,1\}$ by an open path which stays in the box $[-n,n] \times [-m,m] \times \{0,1\}$. 

\begin{thm}\label{thm: mainthm2}
There exists $\e>0$ with the following property. If $n$ and $p$ are such that
\begin{equation}\label{eq: pizza}
\mathbb{P}_p(E(n,n)) > 1-\e\ ,
\end{equation}
then $\theta(p)>0$.
\end{thm}

This theorem will be the main focus of this note. To deduce Theorem~\ref{thm: mainthm1} from it we suppose that $p$ is such that $\theta(p)>0$. Then by exactly the same argument as in \cite[Theorem~8.97]{Grimmett}, we have $\mathbb{P}_p(E(n,n)) \to 1$ as $n \to \infty$. In particular, \eqref{eq: pizza} is satisfied for large enough $n$. Fix such an $n=N$ and note that for $p'<p$ sufficiently close to $p$, \eqref{eq: pizza} will still be satisfied. Thus by Theorem~\ref{thm: mainthm2} we have $\theta(p')>0$ and so $\{p:\theta(p)>0\}$ does not contain $p_c$.

In the next section we will recall arguments that allow us to deduce that if ``rectangles'' of aspect ratio $2$ are crossed with high probability then $\theta(p)>0$. In the following section we will argue using an RSW-type construction that in fact \eqref{eq: pizza} suffices.

\section{Rectangle crossings imply percolation}

In this section we will explain how standard planar percolation arguments can be used to reduce the proof of Theorem~\ref{thm: mainthm2} to an RSW-type estimate. In other words, we will show

\begin{prop}
If $n$ and $p$ are such that
\begin{equation}\label{eq: pizza2}
\mathbb{P}_p(E(2n,n)) \geq 49/50\ ,
\end{equation}
then $\theta(p)>0$.
\end{prop}

\begin{proof}
Let $n \geq 1$. By constructing crossing paths and using the FKG inequality, standard arguments and the thin sandwich property (see Figure~\ref{fig: fig_1} for the construction) show that

\begin{figure}
\setlength{\unitlength}{.5in}
\begin{picture}(10,3)(-5.5,2.5)
\linethickness{2pt}
\put(-3,5){\line(1,0){8}}
\put(-3,3){\line(1,0){8}}
\put(-3,3){\line(0,1){2}}
\put(-1,3){\line(0,1){2}}
\put(1,3){\line(0,1){2}}
\put(3,3){\line(0,1){2}}
\put(5,3){\line(0,1){2}}
\linethickness{1pt}
\put(-3,4){\line(1,0){1}}
\put(-2,4){\line(0,1){.5}}
\put(-2,4.5){\line(1,0){2}}
\put(0,4.5){\line(0,-1){1}}
\put(0,3.5){\line(1,0){1}}
\put(-1,3.25){\line(1,0){3}}
\put(2,3.25){\line(0,1){.5}}
\put(2,3.75){\line(1,0){1}}
\put(1,4.75){\line(1,0){.5}}
\put(1.5,4.75){\line(0,-1){.8}}
\put(1.5,3.95){\line(1,0){2}}
\put(3.5,3.95){\line(0,-1){.25}}
\put(3.5,3.7){\line(1,0){1}}
\put(4.5,3.7){\line(0,1){1}}
\put(4.5,4.7){\line(1,0){.5}}
\linethickness{.5pt}
\put(-.3,5){\line(0,-1){1}}
\put(-.3,4){\line(1,0){1}}
\put(.7,4){\line(0,-1){1}}
\put(2.5,5){\line(0,-1){1.5}}
\put(2.5,3.5){\line(1,0){.25}}
\put(2.75,3.5){\line(0,-1){.5}}
\put(-3,3){\circle*{.15}}
\put(-1,3){\circle*{.15}}
\put(1,3){\circle*{.15}}
\put(5,3){\circle*{.15}}
\put(-3,5){\circle*{.15}}
\put(-3.4,2.6){$(0,0,0)$}
\put(-1.4,2.6){$(n,0,0)$}
\put(.6,2.6){$(2n,0,0)$}
\put(4.6,2.6){$(4n,0,0)$}
\put(-3.4,5.2){$(0,n,0)$}
\end{picture}
\caption{The left two boxes are crossed horizontally by an open path, as are the middle two and the right two. The vertical crossings in the middle two boxes ``glue'' these paths together.}
\label{fig: fig_1}
\end{figure}
\[
\mathbb{P}_p(E(4n,n)) \geq \left[\mathbb{P}_p(E(2n,n))\right]^5\ .
\]
By independence of edge variables in disjoint rectangles,
\[
\mathbb{P}_p(E(4n,2n)^c) \leq \left[ \mathbb{P}_p(E(4n,n)^c)\right]^2\ .
\]
Combining these,
\[
\mathbb{P}_p(E(4n,2n)^c) \leq \left[ 1-\mathbb{P}_p(E(2n,n))^5\right]^2\ .
\]
For $x \in [0,1]$, we have $(1-x^5)^2 \leq 25(1-x)^2$, so for $x \in [49/50,1]$, $(1-x^5)^2 \leq (1/2)(1-x)$. Applying this estimate to $x= \mathbb{P}_p(E(2n,n))$, we see that \eqref{eq: pizza2} implies
\[
\mathbb{P}_p(E(4n,2n)^c) \leq (1/2)\mathbb{P}_p(E(2n,n)^c)\ .
\]
By induction, for each $m \geq 0$,
\begin{equation}\label{eq: pizzapie4}
\mathbb{P}_p(E(2^{m+1}n,2^mn)^c) \leq (1/2^m) \mathbb{P}_p(E(2n,n)^c)\ .
\end{equation}

Now we use a standard construction (see, for instance, \cite[Figure~1]{Jarai} or \cite[Figure~6]{CC1}, \cite{CC2}). Set $B_0 = [0,2n] \times [0,n] \times \{0,1\}$ and for $k \geq 1$,
\[
B_k = \begin{cases}
[0,2^kn] \times [0,2^{k+1}n] \times \{0,1\} & k \text{ odd} \\
[0,2^{k+1}n] \times [0,2^kn] \times \{0,1\} & k \text{ even}
\end{cases}\ .
\]

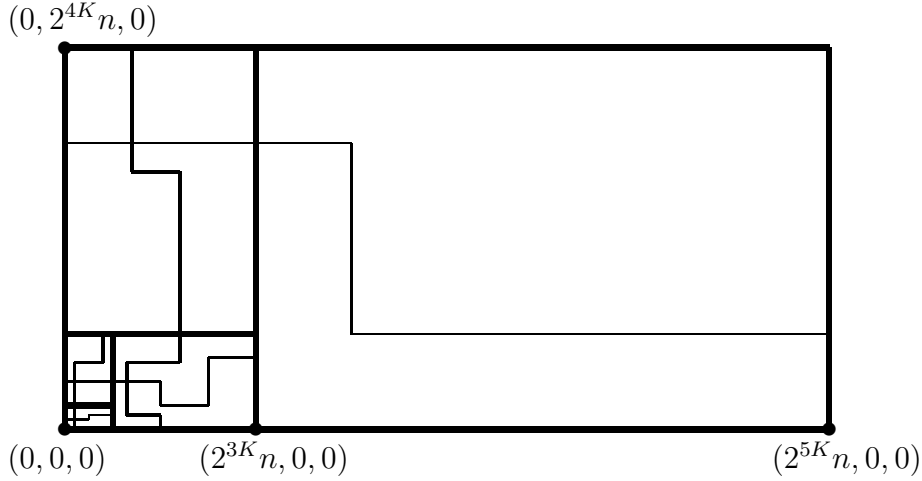
\begin{figure}
\setlength{\unitlength}{.5in}
\begin{picture}(10,5)(-5.5,.5)
\linethickness{2pt}
\put(-3,5){\line(1,0){8}}
\put(-3,1){\line(1,0){8}}
\put(-3,1){\line(0,1){4}}
\put(5,1){\line(0,1){4}}
\put(-1,1){\line(0,1){4}}
\put(-3,2){\line(1,0){2}}
\put(-2.5,1){\line(0,1){1}}
\put(-3,1.25){\line(1,0){.5}}
\linethickness{.5pt}
\put(-3,1.1){\line(1,0){.25}}
\put(-2.75,1.1){\line(0,1){.05}}
\put(-2.75,1.15){\line(1,0){.25}}
\linethickness{1pt}
\put(-2.9,1){\line(0,1){.7}}
\put(-2.9,1.7){\line(1,0){.3}}
\put(-2.6,1.7){\line(0,1){.3}}
\linethickness{.5pt}
\put(-3,1.5){\line(1,0){1}}
\put(-2,1.5){\line(0,-1){.25}}
\put(-2,1.25){\line(1,0){.5}}
\put(-1.5,1.25){\line(0,1){.5}}
\put(-1.5,1.75){\line(1,0){.5}}
\linethickness{1pt}
\put(-2,1){\line(0,1){.15}}
\put(-2,1.15){\line(-1,0){.35}}
\put(-2.35,1.15){\line(0,1){.55}}
\put(-2.35,1.7){\line(1,0){.55}}
\put(-1.8,1.7){\line(0,1){2}}
\put(-1.8,3.7){\line(-1,0){.5}}
\put(-2.3,3.7){\line(0,1){1.3}}
\linethickness{.5pt}
\put(-3,4){\line(1,0){3}}
\put(0,4){\line(0,-1){2}}
\put(0,2){\line(1,0){5}}
\put(-3,1){\circle*{.15}}
\put(-1,1){\circle*{.15}}
\put(5,1){\circle*{.15}}
\put(-3,5){\circle*{.15}}
\put(-3.6,.6){$(0,0,0)$}
\put(-1.6,.6){$(2^{3K}n,0,0)$}
\put(4.4,.6){$(2^{5K}n,0,0)$}
\put(-3.6,5.2){$(0,2^{4K}n,0)$}
\end{picture}
\caption{The event $\cap_{k=K}^{K+4} E_k$ for $K$ odd. The boxes alternate in orientation, the smallest being $B_K$ and working up to $B_{K+4}$. The open crossings (shown) associated to consecutive boxes are forced to touch by the thin sandwich property. Therefore for any $K$, $\cap_{k=K}^\infty E_k$ implies the existence of an infinite open cluster.}
\label{fig: fig_2}
\end{figure}

For $k \geq 0$ and $k$ odd, let $E_k$ be the event that there is an open path in $B_k$ that touches both sides $[0,2^kn] \times \{0\} \times \{0,1\}$ and $[0,2^kn] \times \{2^{k+1}n\} \times \{0,1\}$. For $k$ even we define $E_k$ as the event that there is an open path in $B_k$ that touches both sides $\{0\} \times [0,2^kn] \times \{0,1\}$ and $\{2^{k+1}n\} \times [0,2^kn] \times \{0,1\}$. (See Figure~2 for an illustration of these events.) Because $\mathbb{P}(E_k^c)$ decays exponentially in $k$ (by inequality \eqref{eq: pizzapie4}), with probability one all but finitely many $E_k$'s occur. By the thin sandwich property, this implies the existence of an infinite open path.
\end{proof}

\section{RSW crossing bounds}

For the remainder, we focus on showing that there exists $\e>0$ such that \eqref{eq: pizza} implies \eqref{eq: pizza2}. There are two main difficulties that distinguish this sandwich from the planar case. First, when trying to make connections between paths, it is not enough for their projections onto $\mathbb{Z}^2$ to connect. Indeed, the projections of two non-adjacent sites in the sandwich can be adjacent. For this reason we will have to ``force'' connections at the cost of some constant probability factor. This brings up the second problem, which is that we would like crossing events to hold with probability tending to 1, and the constant probability cost inhibits this. Therefore we will use a bootstrapping argument (in the proof of Proposition~\ref{prop: strongRSW}) that allows to upgrade a weak crossing bound, Proposition~\ref{prop: weakRSW}, to a strong crossing bound.

\begin{lem}\label{lem: crossing}
Suppose that $\theta(p)>0$. Then
\[
\lim_{n \to \infty} \mathbb{P}_p(E(n,n)) = 1\ .
\]
\end{lem}
\begin{proof}
Identical to the one in \cite[Theorem~8.97]{Grimmett}. This argument uses uniqueness of the infinite occupied cluster, which is easily verified for the sandwich $S$.
\end{proof}

\begin{prop}[Weak RSW bound]\label{prop: weakRSW}
Suppose that $\theta(p)>0$. There exists $c_1>0$ such that for all $n$,
\[
\mathbb{P}_p(E(2n,n)) \geq c_1\ .
\]
\end{prop}

\begin{proof}
We will use the method of Bollobas and Riordan \cite{BR}. By gluing paths with the FKG inequality and the thin sandwich property, it suffices to show the above estimate for the probability $\mathbb{P}_p(E(3n/2,n))$. Consider the boxes $S_1 = [0,n] \times [0,n] \times \{0,1\}$, $S_2 = [n/2,3n/2] \times [0,n] \times \{0,1\}$ and $S_3 = [n/2,n] \times [0,n/2] \times \{0,1\}$. Label the left, right, top and bottom sides of $S_1$ as $L_1 = \{0\} \times [0,n] \times \{0,1\}$, $R_1 = \{n\} \times [0,n] \times \{0,1\}$, $T_1 = [0,n] \times \{n\} \times \{0,1\}$ and $B_1 = [0,n] \times \{0\} \times \{0,1\}$. Make similar definitions for $L_i, R_i, T_i$ and $B_i$, $i=2, 3$. See Figure~\ref{fig: fig_3} for an illustration of these definitions.

\begin{figure}
\setlength{\unitlength}{.5in}
\begin{picture}(10,5.5)(-5.5,.5)
\linethickness{1.5pt}
\put(-2,5){\line(1,0){6}}
\put(-2,5){\line(0,-1){4}}
\put(-2,1){\line(1,0){6}}
\put(4,5){\line(0,-1){4}}
\put(0,5){\line(0,-1){4}}
\put(2,5){\line(0,-1){4}}
\put(0,3){\line(1,0){2}}
\put(.8,.5){$S_3$}
\put(1.3,.6){\line(1,0){.7}}
\put(.7,.6){\line(-1,0){.7}}
\put(2,.45){\line(0,1){.3}}
\put(0,.45){\line(0,1){.3}}
\put(-.2,5.4){$S_1$}
\put(.3,5.5){\line(1,0){1.7}}
\put(-.3,5.5){\line(-1,0){1.7}}
\put(2,5.35){\line(0,1){.3}}
\put(-2,5.35){\line(0,1){.3}}
\put(1.8,5.8){$S_2$}
\put(2.3,5.9){\line(1,0){1.7}}
\put(1.7,5.9){\line(-1,0){1.7}}
\put(4,5.75){\line(0,1){.3}}
\put(0,5.75){\line(0,1){.3}}
\put(4.2,3){$R_2$}
\put(.8,3.2){$T_3$}
\linethickness{.5pt}
\put(.4,3){\line(0,-1){1}}
\put(.4,2){\line(1,0){.2}}
\put(.6,2){\line(0,-1){1}}
\multiput(.4,3)(0,.2){5}{\line(0,1){.1}}
\put(.4,4){\line(1,0){.1}}
\multiput(.6,4)(0,.2){5}{\line(0,1){.1}}
\put(.4,2.5){\line(1,0){2}}
\put(2.4,2.5){\line(0,-1){.3}}
\put(2.4,2.2){\line(1,0){.4}}
\put(2.8,2.2){\line(0,1){2}}
\put(2.8,4.2){\line(1,0){1.2}}
\put(-2.4,.6){$(0,0,0)$}
\put(-2,1){\circle*{.15}}
\put(3.6,5.3){$(\frac{3n}{2},n,0)$}
\put(4,5){\circle*{.15}}
\put(2.5,1.2){$U$}
\linethickness{1pt}
\put(2.6,1.5){\vector(0,1){.6}}
\put(-.7,1.5){$R$}
\put(-.4,1.6){\vector(1,0){.9}}
\put(-.7,4.2){$\widetilde R$}
\put(-.4,4.3){\vector(1,0){.9}}
\end{picture}
\caption{The boxes used in the weak RSW bound. The ribbon $R$ crossing $S_3$ is the ``left-most'' open crossing ribbon and the dotted ribbon $\widetilde R$ is its reflection about $T_3$. The ribbon $U$ connecting to $R_2$ is the lowest open crossing ribbon of the region to the right of $R \cup \widetilde R$. At the junction of these two ribbons we build a path $\mathcal{P}(R,U)$ (not shown) to connect the paths contained in the ribbons.}
\label{fig: fig_3}
\end{figure}

We now state a main lemma and describe how Proposition~\ref{prop: weakRSW} follows directly from it. Let $F_n$ be the event that there is an open top-down crossing path of $S_3$ that is connected in $S_2$ by an open path to $R_2$.
\begin{lem}\label{lem: BR}
Suppose that $\theta(p)>0$. For all large $n$,
\[
\mathbb{P}_p(F_n) \geq p^4/4\ .
\]
\end{lem}
Proposition~\ref{prop: weakRSW} follows quickly from this lemma. Indeed, let $G_n$ be the event that there is an open top-down crossing of $S_3$ that is connected in $S_1$ to $L_1$. Also let $H_n$ be the event that there is an open left-right crossing of $S_3$. Using the FKG inequality,
\begin{eqnarray*}
\mathbb{P}_p(E(3n/2,n)) &\geq& \mathbb{P}_p(F_n \cap G_n \cap H_n) \\
&\geq&  \mathbb{P}_p(F_n)^2 ~ \mathbb{P}_p(E(n/2,n/2)))\ ,
\end{eqnarray*}
where the first inequality follows from the thin sandwich property. By Lemmas~\ref{lem: crossing} and~\ref{lem: BR} this is bounded away from zero. This completes the proof.
\end{proof}

\begin{proof}[Proof of Lemma~\ref{lem: BR}]

We first explain the notion of a left-most top-down open crossing of a box (from the appendix). If $P$ is a path in $S_3$ that touches both $T_3$ and $B_3$ (exactly once each) we call $P$ a {\it top-down crossing path} in $S_3$. Define $\pi:S \to \mathbb{Z}^2$ by $\pi(x,y,z) = (x,y)$ and for any set $X \subset S$ set $\pi(X) = \{\pi(v) : v \in X\}$. Also associate to $X$ a {\it ribbon}: $R(X) = \{v \in S: \pi(v) \in \pi(X)\}$.

To a ribbon $R$ of a top-down crossing path in $S_3$, we may associate the region in $S_3$ ``to the left'' of $R$. This is a region of $S_3$ (viewed as a subset of $\mathbb{R}^3$) that does not contain any sites off $R$ that can be connected to the right side of $S_3$ by a lattice path without intersecting $R$ (the region ``to the right'' of $R$). This region is labeled $L(R)$ and the appendix outlines a method to define a ``left-most'' such ribbon; the event that a given ribbon $R$ is the left-most is denoted $D(R)$. We state here two properties of this construction (from Proposition~\ref{prop: ribbon_leftmost}):
\begin{enumerate}
\item The event $A_n$ that there is an open top-down crossing of $S_3$ is the disjoint union $\cup_R D(R)$ over all ribbons of self-avoiding top-down crossings of $S_3$.
\item For each such $R$, the event $D(R)$ depends only on the state of vertices in the region $\overline{L}(R) = R \cup L(R)$.
\end{enumerate}
Any (not necessarily unique) crossing $P$ in the left-most crossing ribbon we call a left-most crossing path.

We now decompose $A_n$ over the events $D(R)$ for the given $n$:
\[
\mathbb{P}_p(A_n) = \sum_{R}\mathbb{P}_p(D(R))\ ,
\]
where the sum is over all top-down crossing ribbons of $S_3$ (not necessarily containing an open path). For any top-down crossing ribbon $R$ of $S_3$, write $\widetilde R$ for the reflection through $T_3$,
\[
\widetilde R = \{(x,y,z) : (x,n-y,z) \in R\}
\]
and let $G(R)$ be the region in $S_2$ strictly ``to the right'' of $R \cup \widetilde R$. That is, $G(R)$ is the set of vertices that can be connected to both $R$ and the right side of $S_2$ by paths (this is a slightly different definition than the left). Define $C(R)$ as the event that there is an open path in $S_2$ that connects the right side $R_2$ to the ribbon $R \cup \widetilde R$. In this event we merely insist that the path contains a vertex in $G(R)$ that is adjacent to $R\cup \widetilde R$, so it is independent of $D(R)$. So we find
\[
\mathbb{P}_p(D(R),C(R)) = \mathbb{P}_p(D(R))~\mathbb{P}_p(C(R)) \geq \mathbb{P}_p(D(R)) \mathbb{P}_p(E(n,n))\ .
\]
Let $C_1(R)$ be the event that there is an open path in $G(R)$ that connects the right side $R_2$ to the ribbon $R$ and $C_2(R)$ the same event but for $\widetilde R$. By symmetry, for fixed $R$, both events $C_1(R)$ and $C_2(R)$ have the same probability. Since these are increasing events,
\[
1-\mathbb{P}_p(E(n,n)) \geq \mathbb{P}_p(C(R)^c) = \mathbb{P}_p(C_1(R)^c \cap C_2(R)^c) \geq \mathbb{P}_p(C_1(R)^c)^2\ ,
\]
so 
\[
\mathbb{P}_p(C_1(R)) \geq 1- \left[ 1-\mathbb{P}_p(E(n,n)) \right]^{1/2} \geq 1/2
\]
for $n$ large enough. Therefore, for $n$ large,
\[
\sum_R \mathbb{P}_p(D(R),C_1(R)) = \sum_R \mathbb{P}_p(D(R)) \mathbb{P}_p(C_1(R)) \geq (1/2) \mathbb{P}_p(A_n) \geq 1/4\ .
\]

In $\mathbb{Z}^2$, the proof would be complete because $D(R) \cap C_1(R)$ would imply the event $F_n$. However, in $S$, a path used in the event $C_1(R)$ does not actually have to meet an open top-down crossing of $S_3$ contained in $R$. To fix this, we need to `hook' the paths together using open sites. Again we must use some independence to do it. Consider again the region $G(R)$ to the right of $R \cup \widetilde R$. We define analogously to before a lowest left-right open crossing $P$ of $G(R)$ and the unique lowest left-right open crossing ribbon $R(P)$ (recall that by `open' we only mean the ribbon contains an open crossing path). We partition the event $D(R) \cap C_1(R)$ according to this ribbon: for any left-right crossing ribbon $U$ of $G(R)$ we write $D(R,U)$ for the event that $U$ is the lowest one. Again, for distinct $U_1$ and $U_2$, the events $D(R,U_1)$ and $D(R,U_2)$ are disjoint. Therefore we write
\[
\mathbb{P}_p(D(R),C_1(R)) = \sum_U \mathbb{P}_p(D(R),D(R,U))\ ,
\]
where the sum is over all left-right crossing ribbons $U$ of $G(R)$ that connect the right side $R_2$ to $R$ (that is, contain a site in $G(R)$ with a neighbor in $R$). For such $R$ and $U$, let $G(R,U)$ be the region of $G(R)$ above $U$ unioned with $\widetilde R$. Note that $D(R)$ and $D(R,U)$ are independent of the state of vertices in $G(R,U)$. We may find a (deterministic) path $\mathcal{P}(R,U)$ in $G(R,U)$ with at most 2 vertices with one neighbor in $R$ and one neighbor in $U$. Let $A(R,U)$ be the event that all vertices on the ribbon of $\mathcal{P}(R,U)$ are open. Note that $D(R) \cap D(R,U) \cap A(R,U)$ implies that $F_n$ occurs and that $\mathbb{P}_p(A(R,U)) \geq p^4$. This implies
\[
\mathbb{P}_p(F_n) \geq p^4 \sum_{R,U} \mathbb{P}_p(D(R),D(R,U)) \geq p^4/4\ .
\]
\end{proof}

For the next corollary, write $Ann(m,n)$ for the annulus of inner radius $m$ and outer radius $n$; that is, $B(n) \setminus B(m)$, where $B(m) = [-m,m]^2 \times \{0,1\}$. 
\begin{cor}\label{cor: pizzapizzapie}
For each $n$ let $\mathcal{A}_n$ be the event that there is an open circuit around the origin in $Ann(n,2n)$. Suppose that $\theta(p)>0$. There exists $c_2>0$ such that for all $n$, 
\[
\mathbb{P}_p(\mathcal{A}_n) \geq c_2\ .
\]
\end{cor}
\begin{proof}
The proof is a standard gluing argument using Proposition~\ref{prop: weakRSW}.
\end{proof}

\begin{prop}[Strong RSW bound]\label{prop: strongRSW}
Suppose that $\theta(p)>0$. Then
\[
\mathbb{P}_p(E(2n,n)) \to 1 \text{ as } n \to \infty\ .
\]
\end{prop}

\begin{proof}
As before, it suffices to show such a limit for the probability $\mathbb{P}_p(E(3n/2,n))$. Define $S_1, S_2$ and $S_3$ as in the proof of Proposition~\ref{prop: weakRSW}. We begin by noting that, just as before, it suffices to show that $\mathbb{P}_p(F_n) \to 1$, where $F_n$ is the event that there is a top-down open crossing of $S_3$ that is connected in $S_2$ by an open path to the right side $R_2$ of $S_2$. To do this, we make the same decomposition as before, writing $A_n$ as the event that there is a top-down open crossing of $S_3$ and 
\[
\mathbb{P}_p(A_n) = \sum_R \mathbb{P}_p(D(R))\ ,
\]
where the sum is over all top-down crossing ribbons of $S_3$. We again denote by $\widetilde R$ the reflected ribbon $\{(x,y,z):(x,n-y,z) \in R\}$ and let $G(R)$ be the region strictly to the right of $R \cup \widetilde R$. If $C_1(R)$ is the event that there is an open path in $G(R)$ that connects the right side $R_2$ to the ribbon $R$, then the same argument as before gives
\begin{equation}\label{eq: pizzapie}
\mathbb{P}_p(C_1(R)) \geq 1-\left[ 1-\mathbb{P}_p(E(n,n)) \right]^{1/2} \to 1 \text{ as } n \to \infty\ .
\end{equation}
We now split the ribbon $R$ into an upper and lower piece:
\[
R_L = R \cap [n/2,n] \times [0,n/4] \times \{0,1\}
\]
\[
R_U = R \cap [n/2,n] \times [n/4,n/2] \times \{0,1\}\ .
\]
Write $C_1(R_U)$ for the event that there is an open path in $G(R)$ that connects the right side $R_2$ to $R_U$ and let $C_1(R_L)$ be the corresponding event for $R_L$. By the FKG inequality,
\[
\mathbb{P}_p(C_1(R)^c) \geq \mathbb{P}_p(C_1(R_U)^c) \mathbb{P}_p(C_1(R_L)^c)\ .
\]
Combining this with \eqref{eq: pizzapie} gives
\[
\mathbb{P}_p(C_1(R_U)^c) \mathbb{P}_p(C_1(R_L)^c) \leq \left[ 1-\mathbb{P}_p(E(n,n)) \right]^{1/2}
\]
and thus
\begin{equation}\label{eq: pizzapie2}
\max\left\{ \mathbb{P}_p(C_1(R_U)),\mathbb{P}_p(C_1(R_L)) \right\} \geq 1- \left[ 1-\mathbb{P}_p(E(n,n))\right]^{1/4}\ .
\end{equation}

Let us suppose that the maximum in \eqref{eq: pizzapie2} is $\mathbb{P}_p(C_1(R_L))$; if this is not the case then we modify the following argument by replacing the lowest open crossing ribbon from $R_2$ to $R_L$ by the highest one from $R_2$ to $R_U$. We write $C_1(R_L)$ as a disjoint union of the events $D(R,U)$ as before, but only using ribbons $U$ in $G(R)$ that connect the right side $R_2$ to the piece of the ribbon $R_L$:
\[
\mathbb{P}_p(C_1(R_L)) = \sum_U \mathbb{P}_p(D(R,U))\ .
\]
Depending on $R$ and $U$ we choose a deterministic vertex $v$ within distance 2 of both $R$ and $U$. Let $Ann_1, \ldots, Ann_k$ be the $k=(1/2) \lfloor \log n \rfloor$ annuli centered at $v$:
\[
Ann_j = B(v;2^{j+1}) \setminus B(v;2^j)\ ,
\]
where $B(v;l)$ is the box $[-l,l]^2 \times \{0,1\}$ translated to be centered at $v$. Let $\widetilde G(R,U)$ be the region of $S_2$ above $U$ and to the right of $R \cup \widetilde R$. Write $H_j(R,U)$ for the event that there is an open path in $Ann_j \cap \widetilde G(R,U)$ that contains vertices $(x_1,y_1,0)$ and $(x_1,y_1,1)$ adjacent to $R$ and vertices $(x_2,y_2,0)$ and $(x_2,y_2,1)$ adjacent to $U$. We claim that there exists $c>0$ such that for all $j$, $n$ and choices of $R$ and $U$,
\begin{equation}\label{eq: pizzapie3}
\mathbb{P}_p(H_j(R,U)) \geq c\ .
\end{equation}
Assuming this for the moment, we note that for any $j$, $H_j(R,U) \cap D(R,U) \cap D(R)$ implies the event $F_n$. Therefore
\begin{eqnarray*}
\mathbb{P}_p(F_n) &\geq& \sum_{R,U} \mathbb{P}_p(D(R), D(R,U), F_n) \\
&\geq& \sum_{R,U} \mathbb{P}_p(D(R),D(R,U),H_j(R,U) \text{ occurs for some } j) \\
&\geq& \left[ 1- c^{\frac{1}{4}\log n} \right] \sum_{R,U} \mathbb{P}_p(D(R),D(R,U)) \\
&=& \left[ 1-c^{\frac{1}{4}\log n} \right] \sum_R \mathbb{P}_p(D(R)) \mathbb{P}_p(C_1(R_L)) \\
&\geq& \left[ 1-c^{\frac{1}{4} \log n} \right] \left[ 1-\left[1- \mathbb{P}_p(E(n,n)) \right]^{1/4} \right] \mathbb{P}_p(E(n,n))\ ,
\end{eqnarray*}
where the third inequality uses independence of the $H_j$'s and the equality uses independence of $D(R)$ and $D(R,U)$. The last quantity converges to 1 as $n \to \infty$. This completes the proof, with the exception of justifying \eqref{eq: pizzapie3}.
\end{proof}

\begin{proof}[Proof of \eqref{eq: pizzapie3}]
Any path in $Ann_j \cap \widetilde G(R,U)$ that contains a vertex adjacent to $R$ and one adjacent to $U$ will be called a crossing quarter-circuit. Just as in the case of left-most open crossings (or crossing ribbons) of a box, we can define the unique innermost open crossing quarter-circuit ribbon. We write $I(R,U,W)$ for the event that $W$ is the innermost such ribbon in $Ann_j \cap \widetilde G(R,U)$. Note as before that for distinct $W_1$ and $W_2$, the events $I(R,U,W_1)$ and $I(R,U,W_2)$ are disjoint. Furthermore the event $I(R,U,W)$ depends only on the state of vertices on $W$ or in its ``interior.'' On this event we may identify (deterministic) paths $\mathcal{P}_1(R,U,W)$ and $\mathcal{P}_2(R,U,W)$ in the ``exterior'' of $W$ such that the first path contains a vertex adjacent to $W$ and one adjacent to $R$, the second path contains a vertex adjacent to $W$ and one adjacent to $U$, and both paths are of length at most 2. Let $J(R,U,W)$ be the event that the ribbons associated to these two paths contain all open sites. 

Note that $J(R,U,W)$ implies the event $H_j(R,U)$. Therefore
\begin{eqnarray*}
\mathbb{P}_p(H_j(R,U)) &=& \sum_W \mathbb{P}_p(I(R,U,W)) \\
&\geq& \sum_W \mathbb{P}_p(I(R,U,W),J(R,U,W)) \\
&\geq& p^8 \sum_W \mathbb{P}_p(I(R,U,W)) \\
&\geq& p^8 \mathbb{P}_p(C_j)\ ,
\end{eqnarray*}
where $C_j$ is the event that there is an open circuit around $v$ in the annulus $Ann_j$. But this probability is bounded away from 0 by Corollary~\ref{cor:  pizzapizzapie}. This completes the proof.
\end{proof}

\appendix
\section{Construction of a left-most crossing}

Let $B(n)$ be the box $[0,n]^2 \times \{0,1\}$ in $S$. We will prove here that one can extract from all top-down open crossing paths of $B(n)$ a unique ``left-most.'' Actually the ribbon of the path, not the path itself, will be unique.

If $P= (x_0, \ldots, x_k)$ is a path in $B(n)$ then recall that $P$ is a top-down crossing of $B(n)$ if $x_0$ is in the top side $[0,n] \times \{n\} \times \{0,1\}$, $x_k$ is in the bottom side $[0,n] \times \{0\} \times \{0,1\}$, and no other vertices of $P$ are in the top or bottom. Recall also the definition of the projection $\pi: S \to \mathbb{Z}^2$ given by
\[
\pi((x,y,z)) = (x,y)\ .
\]
We say that $P$ is \emph{self-avoiding} if $x_i = x_j$ implies $i=j$ and \emph{strongly self-avoiding} if $\pi(x_i) = \pi(x_j)$ implies $|i-j|\leq 1$. Note that if $P$ is strongly self-avoiding it is also self-avoiding. 

An important consequence of the thin sandwich property is the following. 
\begin{lem}
If there is an open top-down crossing of $B(n)$ in $S$ (in a configuration $\omega$) then there is also a strongly self-avoiding open top-down crossing of $B(n)$.
\end{lem}
\begin{proof}
Take any top-down open crossing path $P$ of $B(n)$. We apply the following procedure to remove ``loops.'' Write the vertices of $P$ in order as a sequence $(x_0,x_2, \ldots, x_k)$. If there exists a pair of vertices $x_i,x_j \in P$ with $i<j$ such that $x_i = x_j$ then we remove the segment $x_i, \ldots, x_{j-1}$ from $P$. If instead there exists a pair $x_i,x_j \in P$ with $i<j-1$ and $\pi(x_i) = \pi(x_j)$ but $x_i \neq x_j$ then we remove the segment $x_{i+1}, \ldots, x_{j-1}$ from $P$. By doing this, we produce a new sequence of vertices $P_1$. Repeat the procedure to $P_1$ and continue until there are no longer any of the two types of loops described above.

Note that this procedure must eventually terminate in some sequence $P'$ because at each step we remove at least one vertex from the sequence. We claim that $P'$ is a strongly self-avoiding open top-down crossing of $B(n)$. First, it is a path because we have only removed connected portions $x_k, \ldots, x_l$ such that $x_{k-1}$ and $x_{l+1}$ are neighbors in $S$. Indeed, in the first case above (where $x_i=x_j$ for some $i<j$) then $x_{i-1}$ is a neighbor of $x_j$ if $i>2$ and this is irrelevant when $i=1$. In the second case, $x_i$ and $x_j$ have the same projection and so by the thin sandwich property they are connected. Next, $P'$ still touches the top and bottom sides of the box because we never remove the first and last vertices of $P$. Last $P'$ is by construction strongly self-avoiding: at the end there are no loops to remove.
\end{proof}

We now work with only strongly self avoiding open top-down crossings of $B(n)$. For any such path $P$ we associate to it the subset of $\pi(B(n))$ ``to the left'' of its projection in the following manner. We first project $P$ to a sequence in $\pi(B(n))$:
\[
\pi(P) = (\pi(x_0), \ldots, \pi(x_k))\ .
\]
Note that $\pi(P)$ need not be a path in $\pi(B(n))$ since we may have $\pi(x_i) = \pi(x_{i+1})$ for some $i$. However, after removing duplicate vertices (which must be consecutive) in this sequence we end up with a path $\tilde{\pi}(P)$ which is self-avoiding in $\pi(B(n))$. Further, this path is a self-avoiding top-down crossing of $\pi(B(n))$; that is, its initial vertex is on the top side of $\pi(B(n))$ and its final vertex is on the bottom, with no other of its vertices on the top or bottom. Write $\mathbf{C}(n)$ for the set of all such paths.

For any $\tilde P \in \mathbf{C}(n)$ we consider the continuous path in $\mathbb{R}^2$ obtained by connecting consecutive vertices with line segments. Writing the vertices of the path as $(y_0, \ldots, y_m)$, we then connect the set $[0,n] \times \{n+1/2\}$ to $y_0$ by a vertical line segment and the set $[0,n] \times \{-1/2\}$ to $y_m$ with another vertical line segment. By concatenating this curve with the portion of the boundary $\partial [-1/2,n+1/2]^2$ containing the left side $\{-1/2\} \times [-1/2,n+1/2]$ between the two endpoints of the curve, we obtain a Jordan curve $C_{\tilde P}$; that is, a closed continuous curve with no self-intersections (we will not distinguish between the curve and its trace). By the Jordan curve theorem, the set $\mathbb{R}^2 \setminus C_{\tilde P}$ is composed of two connected components: one bounded and one unbounded. Write $L(\tilde P)$ for the bounded one.

Last, we define the left-most crossing event. Given any $\tilde P \in \mathbf{C}(n)$, write $D(\tilde P)$ for the event that the following two conditions hold.
\begin{enumerate}
\item There is a strongly self-avoiding open top-down crossing $P$ of $B(n)$ such that $\tilde{\pi}(P) = \tilde P$.
\item There is no strongly self-avoiding open top-down crossing $Q$ of $B(n)$ such that the left region $L(\tilde{\pi}(Q))$ is strictly contained in $L(\tilde P)$.
\end{enumerate}

\begin{prop}
The following statements hold.
\begin{enumerate}
\item The event $\mathcal{A}_n$ that there is an open top-down crossing of $B(n)$ is the disjoint union $\cup_{\tilde P \in \mathbf{C}(n)} D(\tilde P)$.
\item For each $\tilde P \in \mathbf{C}(n)$, the event $D(\tilde P)$ depends only on the state of vertices in the region
\[
X(\tilde P) = \pi^{-1} \left( \{v : v \text{ is a vertex of } \tilde P\} \cup \{v : v \in L(\tilde P)\}\right)\ .
\]
($X(\tilde P)$ is the set of vertices which project to $\tilde P$ or to the left of $\tilde P$.)
\end{enumerate}
\end{prop}
\begin{proof}
The second statement of the proposition is obvious from the definition of $D(\tilde P)$, because both $\tilde P$ and any such $Q$ must have all vertices in $X(\tilde P)$, so we prove the first statement. We begin by showing that for distinct $\tilde P$ and $\tilde Q$ in $\mathbf{C}(n)$ the two events $D(\tilde P)$ and $D(\tilde Q)$ cannot hold simultaneously. For a contradiction, assume that they do both hold. The map $\tilde R \mapsto C_{\tilde R}$ is injective, so since $\tilde P \neq \tilde Q$, the curves $C_{\tilde P}$ and $C_{\tilde Q}$ are not equal. Because both curves contain the left boundary $\{-1/2\} \times [-1/2,n+1/2]$, we may start at $(-1/2,-1/2)$, proceeding upward along this boundary, always keeping the regions $L(\tilde Q)$ and $L(\tilde P)$ on our right (here we move ``clockwise''). Continue in this manner until the first point where the curves separate; at this point we have two ways in which to move -- one segment (associated to, say, $C_{\tilde P}$) will move to the ``right'' of the other segment (associated to $C_{\tilde Q}$). In other words, the segment of $C_{\tilde P}$ will be in the region $L_{\tilde Q}$. Thus without loss in generality, there exists a point $p \in C_{\tilde P} \cap L(\tilde Q)$.

Write $\mathbf{D}(n)$ for the set of strongly self-avoiding top-down crossings of $B(n)$ (not necessarily open). Now on the event $D(\tilde P) \cap D(\tilde Q)$, given a point $p \in C_{\tilde P} \cap L(\tilde Q)$, we show how to produce an open path $T \in \mathbf{D}(n)$ such that $L(\tilde{\pi}(T))$ is strictly contained in $L(\tilde Q)$, contradicting the occurrence of the event $D(\tilde Q)$. The point $p$ is in a connected component $C_p$ of $C_{\tilde P} \cap L(\tilde Q)$; this component is just a segment of $C_{\tilde P}$. Order the endpoints of $C_p$ as $z_1$ and $z_2$ by the order in which we meet them while traversing $C_{\tilde Q}$ clockwise. We construct the path $T$ differently depending on the positions of the $z_i$'s. For the construction, let $P,Q$ be open paths in $\mathbf{D}(n)$ such that $\tilde P = \tilde{\pi}(P)$ and $\tilde Q = \tilde{\pi}(Q)$.

First assume that $z_1$ and $z_2$ are contained in $\pi(B(n))$. Then the segment $C_p$ is the projection of some segment $P_p$ of $P$ in $B(n)$ and $z_1$ and $z_2$ are projections of vertices $\hat z_1$ and $\hat z_2$ in $P$. Note that since $Q$ has a first vertex $w_1$ with projection equal to $z_1$ (and last vertex $w_2$ with projection equal to $z_2$), $Q$ is connected to $P_p$. We build $T$ by following $Q$ from the top of $B(n)$ until $w_1$, moving to $\hat z_1$ (if $\hat z_1 \neq w_1$ -- otherwise we do nothing), traversing $P_p$ to $\hat z_2$, then switching back to $Q$ at $w_2$ (if $\hat z_2 \neq w_2$), following it down to the bottom of $B(n)$. Note that since $Q$ is strongly self-avoiding, $z_1 \neq z_2$, so $T$ is also strongly self-avoiding. Because $\tilde{\pi}(T)$ follows $\tilde Q$, only deviating into $L(\tilde Q)$, the region $L(\tilde{\pi}(T))$ is strictly contained in $L(\tilde Q)$ and we have a contradiction.

If either of $z_1$ or $z_2$ is not contained in $\pi(B(n))$ it must be in either $[0,n] \times \{n+1/2\}$ or $[0,n] \times \{-1/2\}$. In this case, the path $P$ either starts or ends (or both) ``to the left'' of $Q$. We modify the above procedure in the obvious way: if $z_1$ is not contained in $\pi(B(n))$ then we build $T$ by starting at the first vertex of $P$, traversing the segment $P_p$ until we reach $Q$ at $\hat z_2$ (if this does not happen then we traverse the whole path $P$ down to the bottom side of $B(n)$, giving that $L(\tilde P)$ is contained in $L(\tilde Q)$, a contradiction) and then switching back to $Q$. In this manner we find an open $T\in \mathbf{D}(n)$ such that $L(\tilde{\pi}(T))$ is strictly contained in $L(\tilde Q)$, a contradiction.

We next show that if the event $\mathcal{A}_n$ occurs then $D(\tilde P)$ must occur for some $\tilde P \in \mathbf{C}(n)$. To do this, we note that the procedure outlined above actually gives a way of constructing an open $T\in \mathbf{D}(n)$ such that $L(\tilde{\pi}(T))$ is strictly contained in $L(\tilde{\pi} (P))$ whenever $P$ is an open path in $\mathbf{D}(n)$ such that there is another open path $Q$ in $\mathbf{D}(n)$ with $\tilde{\pi}(Q)$ intersecting $L(\tilde{\pi}(P))$. Further, the region $L(\tilde{\pi}(T))$ has area no bigger than that of $L(\tilde{\pi}(P))$ minus one. Thus, when $\mathcal{A}_n$ occurs we may start with any open path in $\mathbf{D}(n)$ and apply this procedure (necessarily a finite number of times) until we get a path $P$ such that there is no open $T \in \mathbf{D}(n)$ with $L(\tilde{\pi}(T))$ strictly contained in $L(\tilde{\pi}(P))$. At this point, the event $D(\tilde{\pi}(P))$ occurs. This proves the first item of the proposition.
\end{proof}

Finally, we restate the proposition above in terms of ribbons. It follows immediately from the fact that paths with distinct ribbons must have distinct projections. Given the ribbon $R = R(P) = \{v \in B(n) : \pi(v) \in \pi(P)\}$ of a self-avoiding top-down crossing $P$ of $B(n)$ we write $D(R)$ for the event 
\[
D(R) = \cup_Q D(\tilde{\pi} (Q))\ ,
\]
where the union is over all self-avoiding top-down crossings $Q$ of $B(n)$ with $R(Q) = R$. Let $L(R)$ be the union of all regions $\pi^{-1}\left(L(\tilde{\pi} (Q))\right)$ for all such $Q$.
\begin{prop}\label{prop: ribbon_leftmost}
The following statements hold.
\begin{enumerate}
\item The event $\mathcal{A}_n$ that there is an open top-down crossing of $B(n)$ is the disjoint union $\cup_R D(R)$ over all ribbons of self-avoiding top-down crossings of $B(n)$.
\item For each such $R$, the event $D(R)$ depends only on the state of vertices in the region
\[
\overline{L}(R) = R \cup L(R)\ .
\]
\end{enumerate}
\end{prop}

\bigskip
\noindent
{\bf Acknowledgements.} M. D. thanks C. Newman and the Courant Institute for funds and support while some of this work was done.

\end{document}